\documentclass[12pt,french,a4paper]{article}
\usepackage[utf8]{inputenc}
\usepackage[T1]{fontenc}
\usepackage{amsmath}
\usepackage{amsfonts}
\usepackage{amssymb}
\newcommand{\ds}{\displaystyle}
\newcommand{\Ind}{\mathbb{I}}
\usepackage{graphicx}
\usepackage{epstopdf}
\newtheorem{defn}{D\'efinition}[section]
\newtheorem{defns}{D\'efinitions}[section]
\newtheorem{theo}{Th\'eor\`eme}[section]
\newtheorem{lemm}{Lemme}[section]
\newtheorem{coro}{Corollaire}[section]
\newtheorem{prop}{Proposition}[section]
\newtheorem{exemp}{Exemple}[section]
\newtheorem{remq}{Remarque}[section]
\newenvironment{proof}{\textbf{Preuve : }}{\hfill$\square$}

\begin{document}
\thispagestyle{empty}


\title{Nilpotence dans les alg\`ebres de Malcev\footnote{ \`A la m\'emoire des Professeurs Akry Koulibaly et Artibano Micali.     
}}

\author{C\^ome J. A. B\'ER\'E\footnote{first author email: bere\_jean0@yahoo.fr}, Nakelgbamba B. PILABR\'E\footnote{second author email: pilabrenb@yahoo.fr}\\
 and Moussa OUATTARA\footnote{third author email: ouatt\_ken@yahoo.fr}\\%
Laboratoire T.      N.      AGATA 
       /UFR-SEA\\
        D\'epartement de Math\'ematiques /  
Universit\'e de Ouagadougou%
\\
      Adresse 03 B.      P.      7021 Ouagadougou, Burkina Faso 03\\
}
%
%
\maketitle
\begin{abstract}
The main result  is to prove that if a Malcev algebra  $A$ is \textit{right nilpotent} of degree $n$, then $A$ is \textit{strongly nilpotent} of degree less or equals to $4n^2-2n+1$.\par 
\centerline{\textbf{R\'esum\'e}} \par 
Nous prouvons    que toute alg\`ebre de Malcev $A$ \textit{ nilpotente \`a droite} d'indice $n$ est  \textit{fortement nilpotente} d'un indice inf\'erieur ou \'egal \`a \mbox{$4n^2-2n+1$}.        
\end{abstract}

\textbf{Keywords.} \textit{ Alg\`ebre de Malcev,  nilpotent
\`a droite, nilpotent, fortement nilpotent, indice}.      \par
\textbf{2010 Mathematics Subject Classification :}  17D10 

\section{Introduction}
Lorsque l'alg\`ebre $A$ est de Lie ou alternative, il est facile de montrer que pour un id\'eal $B$, les trois notions de nilpotence, \`a savoir : $B$ est \textit{ nilpotent \`a droite}, $B$ est  \textit{nilpotent} et $B$ est \textit{fortement nilpotent} sont \'equivalentes.      

Pour certaines alg\`ebres non associatives, ce n'est pas  le cas.      Une alg\`ebre de Jordan  peut poss\'eder un id\'eal nilpotent qui n'est pas fortement nilpotent.      De m\^eme l'ag\`ebre de Leibniz de dimension deux  est nilpotente \`a gauche  et non nilpotente \`a droite (\cite[Exemple 3.     2]{beromfpil} ou \cite[Exemple 3.     3]{berpilkob}).      Dans \cite{gmica1, gmica2} M.      Gerlein et A.      Micali, ont montr\'e l'\'equivalence de ces trois notions pour un id\'eal $B$ d'une alg\`ebre de Malcev $A$.      Cependant, certaines d\'emonstrations   comportaient quelques points d'ombres 
 (voir \cite{shest}).     

Gr\^ace \`a l'introduction de nouveaux outils, 
nous montrons que si $A$ est nilpotent \`a droite d'indice $n$, alors  $A$ est fortement nilpotent  d'indice inf\'erieur ou \'egal \`a  $4n^2-2n+1$.     \par 
Nous montrons aussi l'\'equivalence de ces trois notions pour un id\'eal $B$ d'une alg\`ebre de Malcev $A$ qui est $J_k$-nil.     


Nous commen\c cons par  un rappel de quelques d\'efinitions et exemple sur les alg\`ebres de Malcev, puis dans la section \ref{motdroit} nous \'etablissons quelques r\'esultats sur les produits droits de longueurs $n$.  
 La section \ref{poids} est d\'edi\'ee \`a l'\'etude des produits droits de poids $n$. 
 Dans la section \ref{mainresult}, nous \'enoncons un th\'eor\`eme dans le cas d'un id\'eal $J_k$-nil.      Puis nous montrons que les trois types de nilpotence sont \'equivalents pour  une alg\`ebre de Malcev $A$ (cf.      le corollaire \ref{equiv}).

\section{Pr\'eliminaires}\label{prelim}
\begin{defn}
 \label{def1.1}Par la suite, sauf mention expresse du
contraire, $K$ d\'esignera un corps commutatif de caract\'eristique diff\'erente
de 2 et tout espace vectoriel sera  de dimension finie
sur $K$.      Si $A$ est une $K$-alg\`ebre, non n\'ecessairement associative,
l'application $K$-trilin\'eaire $J:A\times A\times A\longrightarrow A$
d\'efinie par $(x,y,z)\longmapsto(xy)z+(yz)x+(zx)y$ est appel\'ee le \textit{jacobien}
de $A$.     \par
Soient $U,V,W$ trois sous alg\`ebres de $A$.      $J\left(U,V,W\right)$ est le sous espace vectoriel de $A$ engendr\'e par les \'el\'ements de la forme $J\left(u,v,w\right)$ o\`u $u\in U$, $v\in V$ et $w\in W$.     \par
 On dira que $A$ est une \textit{alg\`ebre de Malcev} si
l'une quelconque des trois  conditions \'equivalentes suivantes est v\'erifi\'ee : 
\begin{description}
\item [1.] $x^{2}=0$ pour tout $x$ dans $A$ et $J(x,y,xz)=J(x,y,z)x$,
quel que soient $x,y,z$ dans $A$ ; 
\item [2.] $x^{2}=0$ pour tout $x$ dans $A$ et $J(x,xy,z)=J(x,y,z)x$,
quel que soient $x,y,z$ dans $A$ ; 
\item [3.] $x^{2}=0$ pour tout $x$ dans $A$ et $(xy)(xz)=((xy)z)x+((yz)x)x+((zx)x)y$,\\
 quels que soient $x,y,z$ dans $A$.

Il est clair que ces d\'efinitions et \'equivalences ne d\'ependent pas
de la caract\'eristique de $K$ mais si celle-ci est diff\'erente de 2,
les conditions ci-dessus mentionn\'ees sont encore \'equivalentes \`a la
suivante :

\item [4.] $x^{2}=0$ et $(xz)(yt)=((xy)z)t+((yz)t)x+((zt)x)y+((tx)y)z,$
quels que soient $x,y,z,t$ dans $A$.

Si la caract\'eristique de $K$ est diff\'erente de 2, on  a l'\'egalit\'e
(cf.      \cite{bere1}) :

\item [5.] $J\left(x,y,z\right)t=J\left(t,x,zy\right)+J\left(t,y,xz\right)+J\left(t,z,yx\right),$
quels que soient $x,y,z,t$ dans $A$.      
\end{description}
\end{defn}

Si $A$ est une $K$-alg\`ebre, consid\'erons l'alg\`ebre not\'ee $A^{-}$
dont l'espace vectoriel sous-jacent co\"\i ncide avec $A$ et dont la
multiplication est d\'efinie par $[x,y]=xy-yx$ pour $x,y$ parcourant
$A$.      On v\'erifie sans peine, que si $A$ est une alg\`ebre associative,
l'alg\`ebre $A^{-}$ est de Lie et si $A$ est une alg\`ebre alternative,
l'alg\`ebre $A^{-}$ est de Malcev.      De plus, le jacobien d'une alg\`ebre
de Lie \'etant nul, toute alg\`ebre de Lie est de Malcev.      \par

\begin{exemp}
Soit $A$ la $K$-alg\`ebre de dimension 4 dont
la table de multiplication relative \`a une base $\{e_{1},e_{2},e_{3},e_{4}\}$
s'\'ecrit 
$e_1e_2=e_1=-e_2e_1$, $e_1e_2=e_1=-e_2e_1$, $e_3e_1=e_4=-e_1e_3$, $e_3e_2=e_3=-e_2e_3$, $e_2e_4=e_4=-e_4e_2$ et tous les autres produits \'etant nuls. 
On v\'erifie que $A$ est une alg\`ebre
de Malcev (par calcul direct), non de Lie si la caract\'eristique de $K$ est aussi diff\'erente
de $3$ (car $J(e_{1},e_{2},e_{3})=-3e_{4}$).      Si $K$ est un corps de
caract\'eristique $3$, cette alg\`ebre est de Lie. 
\end{exemp}
 On renvoie \`a \cite{akry, malek} pour les renseignements
compl\'ementaires concernant les alg\`ebres de Malcev.

Trois types de nilpotence d'un id\'eal $B$ d'une alg\`ebre de Malcev $A$ ont \'et\'e introduits par
A.      Micali et Ch.      Gerlein dans \cite{gmica1,gmica2}.      Rappelons les :

\addtocounter{defns}{1}
\addtocounter{defn}{1}
\begin{defns}
Soit un id\'eal $B$ d'une alg\`ebre de Malcev $A$, on introduit
les notations et terminologies suivantes :\par      
Soit $P$ un produit de $m$ facteurs $s_{m},s_{m-1},\cdots,s_{1}$
distincts ou non, associ\'es d'une mani\`ere quelconque et dont $n$   ($n\leq m$) ou
plus de ses facteurs appartiennent \`a $B$.      
On dira que le produit $P$ est de longueur $m$ et de poids $n$
relativement \`a $B$ ou plus simplement que $P$ est de longueur $m$ et de poids $n$ dans $B$.      La longueur $m$ de $P$ sera not\'ee $\#\left(P\right)$
et son poids $n$ relativement \`a $B$ sera not\'ee $\#_{B}\left(P\right)$.     
\par
 Lorsque $P=\left(\left(\cdots\left(\left(s_ms_{m-1}\right)s_{m-2}\cdots\right)s_3\right)s_2\right)s_1$  
 o\`u l'association
est faite \`a droite syst\'ematiquement, on dira que $P$ est un \textit{produit droit} et on \'ecrira alors simplement $P=s_ms_{m-1}s_{m-2}\cdots s_{1}$, sans parenth\`eses.\\ 
Soient $S_1,S_2,\cdots,S_p$ des produits droits. On peut en faire le produit \`a droite $N=S_1S_2\cdots S_p$.
 On dira que $N$ est un produit normal.    
\begin{itemize}
\item  $B^{n}$ le sous-espace vectoriel de l'alg\`ebre $A$ engendr\'e par tous
les produits droits de longueur $n$ et de poids $n$ relativement
\`a $B$.      Par convention, on posera $B^{0}=A$.     
Et on dit que l'id\'eal
$B$ est \textit{nilpotent \`a droite} s'il existe un entier $n\geq1$
tel que $B^{n}=\left\{ 0\right\} $.     \par
 En particulier $A^{n}$ est engendr\'e
par tous les produits droits de longueur $n$ et c'est un id\'eal de l'alg\`ebre $A$.      
\item  $B^{\{n\}}$ le sous-espace vectoriel de $A$ engendr\'e par les produits
de $n$ \'el\'ements de $B$ associ\'es d'une fa\c con quelconque.      Et on dit
que l'id\'eal $B$ est \textit{nilpotent} s'il existe un entier $n\geq1$
tel que $B^{\{n\}}=\left\{ 0\right\} $.      
\item  $B^{\left\langle n\right\rangle }$ l'ensemble des sommes de produits
d'\'el\'ements de $A$ avec au moins $n$ \'el\'ements de $B$.      On voit que
$B^{\left\langle n\right\rangle }$ est un id\'eal de $A$ et on a\\   $A\supseteq B^{\left\langle 1\right\rangle }\supseteq B^{\left\langle 2\right\rangle }\supseteq\cdots\supseteq B^{\left\langle n\right\rangle }\supseteq\cdots$
et $B^{\left\langle i\right\rangle }B^{\left\langle j\right\rangle }\subseteq B^{\left\langle i+j\right\rangle }$
quels que soient les entiers $i,j\geq1$.      On dit que l'id\'eal $B$
est \textit{fortement nilpotent} s'il existe un entier $n\geq1$ tel
que $B^{\left\langle n\right\rangle }=\left\{ 0\right\} $.      

D'une mani\`ere g\'en\'erale si $A$ est une alg\`ebre non associative qui n'est ni commutative, ni anticommutative, nous ajoutons la d\'efinition suivante :\par
 Lorsque $P=s_1\left(s_2\left(s_3\left(\cdots s_{m-2}\left(s_{m-1}s_m\right)\right)\cdots\right)\right)$ 
 o\`u l'association
est faite \`a gauche syst\'ematiquement, on dira que $P$ est un \textit{produit gauche}.     
 \item  $^{n}\!B$ le sous-espace vectoriel de l'alg\`ebre $A$ engendr\'e par tous les produits gauches de longueur $n$ et de poids $n$ relativement
\`a $B$.      Par convention, on posera $^0\!B=A$.      Et on dit que l'id\'eal
$B$ est \textit{nilpotent \`a gauche} s'il existe un entier $n\geq1$
tel que $^n\!B=\left\{ 0\right\} $.     \par
En particulier si $A$ est commutative ou anticommutative, on a $^n\!B$ co\"\i ncide avec $B^n$.     
\end{itemize}
\end{defns}

La preuve  qu'une alg\`ebre $A$ qui est \textit{ nilpotente \`a droite}
est une alg\`ebre \textit{fortement nilpotente} 
donn\'ee dans \cite{gmica1} n'est pas assez convainquante (cf.      \cite{shest}).      C'est pourquoi
nous d\'eveloppons d'autres outils qui vont nous permettre de re\'ecrire une preuve
plus rigoureuse.

\begin{defn} Soit $D$ un sous espace vectoriel d'une alg\`ebre de
Malcev $A$.      Soit $k$ un entier sup\'erieur ou \'egal \`a $1$.      Notons $D_{(A,k)}=D\underset{k\textnormal{ facteurs}}{\underbrace{AAA\dots A}}$
le sous espace vectoriel engendr\'e par tous les produits droits de
la forme $da_{k}\cdots a_{3}a_{2}a_{1}$ o\`u $d\in D$ et $a_{i}\in A$
pour tout entier $i$, tel que $1\leq i\leq k$.      
\end{defn}
\begin{defn} Soit $B$ un id\'eal non nul d'une alg\`ebre de Malcev $A$.     
Si l'id\'eal $J\left(B,A,A\right)$ satisfait pour un entier $k$ sup\'erieur ou \'egal \`a $1$,
la relation $J\left(B,A,A\right)_{(A,k)}=\left\{ 0\right\} $,%
on dira que $B$ est \textit{$J_{k}$-nil} dans $A$.       
\end{defn}

\begin{defn} Soient $A$ une alg\`ebre de Malcev et $B$ un id\'eal 
de $A$.      Soient $a,b$ deux \'el\'ements de A.      On dira que $a$ et $b$
sont \'egaux modulo $B$ si leur diff\'erence est dans $B$.      
\end{defn}

\begin{defn} Soit $n$ un entier naturel non nul. On pose par d\'efinition 
$\Ind(n)=\{1,2,\cdots,n\}$.
\end{defn}

\section{Produits droits dans l'alg\`ebre de Malcev $A$}\label{motdroit}

La preuve du lemme \ref{drf} se trouve dans \cite{bere1}, mais nous la reprenons.
\begin{lemm}\label{drf}  Soit $A$ une alg\`ebre de Malcev de dimension
finie et $B$ un id\'eal de $A$.      Posons $B_{0}=A$, $B_{1}=B$ et
$B_{k}=B^{k}+J(B,A,A)$ pour tout entier naturel $k$ sup\'erieur ou \'egal \`a $2$ ; $B_{k}$
est un id\'eal de $A$ v\'erifiant $B_{k}\supseteq B_{k+1}$.
\end{lemm}      
\begin{proof} Il est bien connu que $J(B,A,A),B_{0},B_{1}$ sont
des id\'eaux.      Supposons que pour un entier $k\geq2$, $B_{k}$ soit
un id\'eal, alors on a $B_{k}.     A\subseteq B^{k}+J(B,A,A)$.      Montrons
que $B_{k+1}$ est aussi un id\'eal.      En effet on a ; 
\begin{align*}
B_{k+1}.     A & =\left(B^{k+1}+J(B,A,A)\right).     A\\
 & \subseteq B^{k+1}.     A+J(B,A,A)\\
 & \subseteq\left(B^{k}.     B\right).     A+J(B,A,A)\\
 & \subseteq J(B^{k},B,A)+\left(B.     A\right).     B^{k}+\left(A.     B^{k}\right).     B+J(B,A,A)\\
 & \subseteq B.     B^{k}+\left(B^{k}+J(B,A,A)\right).     B+J(B,A,A)\\
 & \subseteq B^{k+1}+J(B,A,A)=B_{k+1}.     
\end{align*}
 \end{proof}

\begin{prop}\label{prop:P-som} 
Soient un produit droit $P_{0}=a_{m}a_{m-1}a_{m-2}\cdots a_{3}a_{2}a_{1}$   et 
$Q_{0}$ un produit quelconque de longueur $m'$. 
 Posons  
pour tout entier $0\leq i\leq m-1$, $P_{i}=a_ma_{m-1}\cdots a_{i+1}=\ds\prod_{j=1}^{m-i}a_{m-j+1}$ et pour tout entier $0\leq r\leq m-2$,  $Q_{r+1}=a_{r+1}Q_{r}$.
Alors : 
$$
T_{m}=Q_{0}P_{0}=\sum_{i=1}^{m-1}Q_{i-1}P_{i}a_{i}+Q_{m-1}a_{m}-\ds\sum_{i=1}^{m-1}J\left(Q_{i-1},P_i,a_{i}\right).     $$
\end{prop}
\begin{proof}  Avant de faire la preuve fixons quelques d\'efinitions de produits :

$m$ d\'esignera  la longueur de $P$.\\ 
Posons pour $1\leq i\leq m-1$, $Q'_{i-1}=Q_{i}$, $a'_{m-i+1}=a_{m-i+2}$. Nous avons alors,  $$P'_i=\ds\prod_{j=1}^{m-i}a'_{m-j+1}=\ds\prod_{j=1}^{m-i}a_{m-j+2}=\prod_{j=1}^{m+1-(i+1)}a_{m+1-j+1}=P_{i+1}.$$
Il est clair que si   $m=2$, on a :\\
$\begin{aligned} T_{2}&=Q_{0}\left(a_{2}a_{1}\right) =Q_{0}a_{2}a_{1}-Q_0a_1a_2-J\left(Q_{0},a_{2},a_{1}\right),\textnormal{ et si  }    m=3:\\
T_{3}&=Q_{0}\left(a_{3}a_{2}a_{1}\right)\\ 
& =  Q_{0}\left(a_{3}a_{2}\right)a_{1}+a_{1}Q_{0}\left(a_{3}a_{2}\right)-J\left(Q_{0},a_{3}a_{2},a_{1}\right)\\
 & =  Q_{0}\left(a_{3}a_{2}\right)a_{1}+Q_{1}\left(a_{3}a_{2}\right)-J\left(Q_{0},a_{3}a_{2},a_{1}\right)\\
 & = Q_{0}\left(a_{3}a_{2}\right)a_{1}+Q_{1}a_{3}a_{2}+a_{2}Q_{1}a_{3}-J\left(Q_{1},a_{3},a_{2}\right)-J\left(Q_{0},a_{3}a_{2},a_{1}\right)\\
 & = Q_0P_1a_1+Q_1a_3a_2+Q_2a_3-J\left(Q_{1},a_{3},a_{2}\right)-J\left(Q_{0},a_{3}a_{2},a_{1}\right).
\end{aligned}
$\par

Posons en hypoth\`ese que pour un produit $P$ de longueur $m$ on
a :\par
 \begin{equation}\label{SionPm}
     T_{m}=\sum_{i=1}^{m-1}Q_{i-1}P_{i}a_{i}+Q_{m-1}a_{m}-\ds\sum_{i=1}^{m-1}J\left(Q_{i-1},P_i,a_{i}\right).     
    \end{equation} 

Alors on aura pour un produit $P=a_{m+1}a_{m}a_{m-1}\cdots a_{3}a_{2}a_{1}$ de longueur $m+1$ :\par 
$\begin{aligned}T_{m+1} & =  Q_{0}\ds\prod_{k=1}^{m+1}a_{m-k+2}
 =Q_{0}\left[\ds\prod_{k=1}^{m}a_{m-k+2}a_{1}\right]\\ 
 & = Q_{0}\ds\prod_{k=1}^{m}a_{m-k+2}a_{1}+a_1Q_0\ds\prod_{k=1}^{m}a_{m-k+2}-J\left(Q_0,\ds\prod_{k=1}^{m}a_{m-k+2},a_{1}\right)\\
 & = Q_{0}\ds\prod_{k=1}^{m}a_{m-k+2}a_{1}+Q_{1}\ds\prod_{k=1}^{m}a_{m-k+2}-J\left(Q_{0},\ds\prod_{k=1}^{m}a_{m-k+2},a_{1}\right) 
\end{aligned}$\par 


Il s'ensuit alors que le produit $T_{m+1}$ vaut,  \par

\begin{align}T_{m+1} & =  Q_{0}\ds\prod_{k=1}^{m+1}a_{m-k+2}
 =Q_{0}\left[\ds\prod_{k=1}^{m}a_{m-k+2}a_{1}\right]\\
 &=Q_{0}\left(P'_0a_{1}\right)\\
 & = Q_{0}P'_0a_{1}+Q_{1}P'_0-J\left(Q_{0},P'_0,a_{1}\right) \\
& = Q_{0}P_1a_{1}+Q'_0P'_0-J\left(Q_{0},P_1,a_{1}\right) \label{equ3}
\end{align}\par 

Par application de l'hypoth\`ese de r\'ecurrence (\ref{SionPm})  \`a $Q'_0P'_0$ (car $P'_0$ est de longueur $m$) :

$\begin{aligned}
Q'_0P'_0 & =\sum_{i=1}^{m-1}Q'_{i-1}P'_ia'_{i}+Q'_{m-1}a'_{m}-\ds\sum_{i=1}^{m-1}J\left(Q'_{i-1},P'_i,a'_{i}\right)\\
 & =\sum_{i=1}^{m-1}Q_{i}P_{i+1}a_{i+1}+Q_{m}a_{m+1}-\ds\sum_{i=1}^{m-1}
J\left(Q_{i},P_{i+1},a_{i+1}\right)
\end{aligned}$\par

De l'\'equation (\ref{equ3}), il vient que :

$\begin{aligned}
T_{m+1} & = Q_{0}P_1a_{1}+Q'_0P'_0-J\left(Q_{0},P_1,a_{1}\right)\\
	& = Q_{0}P_1a_{1}+\left[\sum_{i=1}^{m-1}Q_{i}P_{i+1}a_{i+1}+Q_{m}a_{m+1}\right.     \\
	& \left.     \hphantom{+Q_{m}a_{m+1}AAAAAA}-\ds\sum_{i=1}^{m-1}
J\left(Q_{i},P_{i+1},a_{i+1}\right)\right]-J\left(Q_{0},P_1,a_{1}\right)\\
 & = \sum_{i=1}^{m}Q_{i-1}P_{i}a_{i}+Q_{m}a_{m+1}-\ds\sum_{i=1}^{m}
J\left(Q_{i-1},P_{i},a_{i}\right).
\end{aligned}$
\end{proof}

\begin{remq}\label{remq}
Sous  les hypot\`eses de la proposition \ref{prop:P-som} et  en 
%
consid\'erant le tableau suivant, il vient que :

\begin{table*}[h]
\begin{tabular}{||l||l||l||l||l}
\hline 
 $P_{m-1}=a_{m}$  & $P_{m-2}=P_{m-1}a_{m-1}$  & $\cdots$  & $P_{i-1}=P_{i}a_{i}$  & $\cdots$ \tabularnewline
\hline 
 $Q_{m-1}=a_{m-1}Q_{m-2}$  & $Q_{m-2}=a_{m-2}Q_{m-3}$  & $\cdots$  & $Q_{i-1}=a_{i-1}Q_{i-2}$  & $\cdots$ \tabularnewline
\hline 
\end{tabular}
\end{table*}

\begin{table*}[h]%
\begin{tabular}{l||l||l||l||l||l||}
\hline 
$\cdots$  & $P_{i+j}=P_{i+j+1}a_{i+j+1}$  & $\cdots$  & $P_{2}=P_{3}a_{3}$  & $P_{1}=P_{2}a_{2}$  & $P_{0}=P_{1}a_{1}$ \tabularnewline
\hline 
$\cdots$  & $Q_{i+j}=a_{i+j}Q_{i+j-1}$  & $\cdots$  & $Q_{2}=a_{2}Q_{1}$  & $Q_{1}=a_{1}Q_{0}$  & $Q_{0}=Q_{0}$ \tabularnewline
\hline 
\end{tabular}\label{alg5} 
\end{table*}

L'ensemble $\Lambda=\left\{(a_i)_{1\leq i\leq m},(b_j)_{1\leq j\leq m'}\right\}$ est l'ensemble des facteurs distincts ou non qui donne le produit $P$. On v\'erifie facilement qu'il donne \'egalement les produits $Q_{k-1}P_ka_k,Q_{m-1}a_m$ o\`u $k\in\Ind(m-1)$. 
Il vient alors que pour $k\in\Ind(m-1)$ :
\begin{eqnarray}
\#\left(P\right)\phantom{_B}&=\#\left(Q_{k-1}\right)+\#\left(P_k\right)+1\label{rem1.1}\\
\#_{B}\left(P\right)&=\#_{B}\left(Q_{k-1}\right)+\#_{B}\left(P_k\right)+
 \#_{B}\left(a_k\right)\label{rem1.2} \\
 \#\left(P\right)\phantom{_B}&=\#\left(Q_{m-1}\right)+1\label{rem1.3}\\
\#_{B}\left(P\right)&=\#_{B}\left(Q_{m-1}\right)+
 \#_{B}\left(a_m\right).\label{rem1.4}
\end{eqnarray}     
\end{remq}

\begin{lemm} \label{lem:nilpo1} Tout produit $T$ de longueur $m$
dans une alg\`ebre de Malcev $A$ de dimension finie est combinaison
lin\'eaire de produits droits de longueur $m$ modulo l'id\'eal $J(A,A,A)$.     
\end{lemm} 

\begin{proof}
Proc\'edons par r\'ecurrence sur la longueur $m$.      Lorsque
la longueur $m$ est inf\'erieure ou \'egale \`a $3$, c'est \'evident que
le r\'esultat est vrai.      Supposons alors la relation vraie pour tout
produit de longueur strictement inf\'erieure \`a $m\geq4$.     

Nous ferons la d\'emonstration en consid\'erant $T$ comme un produit
$Q_{0}P_{0}$ avec $P_{0}$ un produit droit de longueur $n>1$.     \\
 Regardons ce qui se passe avec un produit $T=Q_0P_0$ de longueur $m>n$.     \\
 Alors d'apr\`es la proposition \ref{prop:P-som}, 
\[
T=Q_{0}P_{0}=\sum_{i=1}^{n-1}Q_{i-1}P_{i}a_{i}
+Q_{n-1}a_{n}-\sum_{i=1}^{n-1}J\left(Q_{i-1},P_{i},a_{i}\right).     
\]
Les produits suivants $Q_{i-1}P_{i}$ pour $1\leq i\leq n-1$ et $Q_{n-1}$
sont des produits de longueur \'egale \`a $m-1$ et sont donc des combinaisons
lin\'eaires de produits droits de longueur $m-1$ modulo l'id\'eal $J(A,A,A)$.     
Ceci montre que $T$ est combinaison lin\'eaire de produits
droits de longueur $m$ modulo l'id\'eal $J(A,A,A)$.      
\end{proof}

\section{Produits droits de poids $n$ dans un id\'eal $B$}\label{poids}

\begin{lemm} \emph{\label{lem:-prod-normo}}Tout produit de longueur
$m$ et de poids $n$ relativement \`a $B$ dans une alg\`ebre de Malcev
est combinaison lin\'eaire de produits normaux de longueur $m$ et de
poids $n$.
\end{lemm}       

\begin{proof} Le lemme est \'evident si la longueur $m$ du produit  est
inf\'erieure ou \'egale \`a trois.      
Si $m=4$, alors il existe quatre \'el\'ements $a_{1},a_{2},a_{3},a_{4}$
dans $A$ tels que $P$ est \'egal \`a l'une des combinaisons de produits
droits suivants : $$a_{4}a_{3}a_{2}a_{1}\textnormal{ et }\left(a_{4}
a_{3}\right)\left(a_{2}a_{1}\right)
=a_{4}a_{3}a_{2}a_{1}+a_{1}a_{4}
a_{3}a_{2}+a_{2}a_{1}a_{4}a_{3}
+a_{3}a_{2}a_{1}a_{4}. $$    

Supposons le lemme vrai pour tout produit de longueur inf\'erieure ou
\'egal \`a $m$.      Soit $P$ un produit de longueur $m+1$.      Alors $P$ s'\'ecrit :\\
\noindent Soit $P_{0}a_{0}$ avec la longueur de  $P_0$ \'egal \`a $m$,\\ 
\noindent Soit $\left(P_{4}P_{3}\right)\left(P_{2}P_{1}\right)=P_{4}P_{2}P_{3}P_{1}+P_{1}P_{4}P_{2}P_{3}+P_{3}P_{1}P_{4}P_{2}+P_{2}P_{3}P_{1}P_{4}$
 avec  la longueur de chaque $P_i$ dans    $\Ind(m-1)$.\par       
 L'hypoth\`ese
de r\'ecurrence montre que $P_0$, $P_4P_2P_3$, $P_1P_4P_2$, $P_3P_1P_4$, $P_2P_3P_1$ sont combinaisons
lin\'eaires de produits normaux conservant les longueur et poids initiaux.      Ainsi $P$ l'est aussi.    Il est \'evident que par construction des termes  la longueur et le poids de $P$ sont conserv\'es par chacun des termes.
\end{proof}

\begin{lemm}\label{lienon-1} Soient A une alg\`ebre de Malcev
et $B$ un id\'eal non nul de $A$.      Tout produit $T$ de $A$ de longueur
$m\ge1$ et de poids $n\ge1$ relativement \`a $B$ est combinaison
lin\'eaire de produits droits de longueur $m$ et de poids $n$ relativement
\`a $B$, modulo l'id\'eal $J(B,A,A)$. 
\end{lemm}     

\begin{proof} 
Le lemme est \'evident si 
 $T$ est de longueur  $m\leq3$.     
 \par 
Posons l'hypoth\`ese suivante : 
la relation est vraie pour tout produit $T$ dont la longueur
est strictement inf\'erieure \`a $m$.\par 
Soit $Q_0P_0$ un produit normal
de poids $n$ relativement \`a $B$ et de longueur $m=p'+p$ o\`u $Q_{0}$
(respectivement $P_{0}$) est un produit droit de longueur $p'$ (respectivement
$p$).      Alors d'apr\`es la proposition \ref{prop:P-som}, 
$$
Q_0P_0=\ds\sum_{i=1}^{p-1}Q_{i-1}P_{i}a_{i}+Q_{p-1}a_{p}-
\sum_{i=1}^{p-1}J\left(Q_{i-1},P_{i},a_{i}\right).     
$$

Les produits  $Q_{i-1}P_{i}$ ,
 $Q_{p-1}$  ( o\`u $1\leq i\leq p-1$)
sont des produits de longueur \'egale \`a $m-1$.\\
Ils sont donc des combinaisons  
lin\'eaires de produits droits de longueur $m-1$ modulo l'id\'eal $J(B,A,A)$.     \\
D'o\`u $Q_{0}P_{0}$ est combinaison lin\'eaire de produits
droits de longueur $m$ modulo l'id\'eal $J(B,A,A)$.     

Grace aux \'equations \ref{rem1.1}-\ref{rem1.4}
, il est clair que pour $i\in\Ind(p-1)$, le poids relativement
\`a $B$ du produit normal $Q_{i-1}P_{i}a_{i}$ est $n$. Il en est de m\^eme pour le poids du produit $Q_{p-1}a_{p} $, relativement
\`a $B$.
\end{proof}

\begin{lemm} \label{Bn} Soient $A$ une alg\`ebre de Malcev $A$ de dimension
finie et $B$ un id\'eal non nul de $A$.      Soit $P_{0}=a_{m}a_{m-1}a_{m-2}\cdots a_{3}a_{2}a_{1}$
un produit droit de longueur $m$ et de poids sup\'erieur ou \'egal \`a
$n\geq1$, relativement \`a $B$.      Alors $P_{0}$ appartient \`a l'id\'eal
$B_{n}$. 
\end{lemm}

\begin{proof} Soit $\sigma$ une injection croissante de $\mathbb{I}(n)$
dans $\mathbb{I}(m)$ telle que pour
tout $j\in\mathbb{I}(n),a_{\sigma(j)}\in B$.      Posons pour tout entier
$1\leq j\leq n-1$, 
$$\begin{aligned}
Q'_{0}=&a_{m}a_{m-1}\cdots a_{\sigma(n)+1},\\ 
Q_{j}=&Q'_{j-1}a_{\sigma(n-j+1)}, \\                   
Q'_{j}=&Q_{j}a_{\sigma(n-j+1)-1}\cdots a_{\sigma(n-j)+1},\\ 
Q_{n}=&Q'_{n-1}a_{\sigma(1)}\cdots a_{1}.                       
\end{aligned}$$
 
Il est clair que, $Q_{1}\in B^{1}\subseteq B^{1}+J(B,A,A)=B_1\textnormal{ et }Q'_{1}\in B$ ;\\
 $Q_{2}=Q'_{1}a_{\sigma(n-j+1)}\in B.     B\subseteq B^{2}+J(B,A,A)=B_{2}$.     \\
 Supposons que pour $1\leq j<n$, $Q_{j}\in B_{j}=B^{j}+J(B,A,A)$
et montrons que $Q_{j+1}\in B_{j+1}=B^{j+1}+J(B,A,A)$.     \\
En exploitant
l'hypoth\`ese, on a que\\
 $Q'_{j}=Q_{j}a_{\sigma(n-j+2)-1}\cdots a_{\sigma(n-j+1)+1}$ appartient
\`a l'id\'eal $B_{j}$.      Par suite\\
 $Q_{j+1}=Q'_{j}a_{\sigma(n-j+1)}\in B_{j}.     B\subseteq B^{j+1}+J(B,A,A)=B_{j+1}$.      \\
Alors $P_{0}=Q_{n}\in B_{n}$.     \hfill $\square$
\end{proof}

\begin{lemm}\label{lem:laqt} Soient $A$ une alg\`ebre de Malcev  %
 de dimension finie et $B$ un id\'eal $J_{k}$-nil dans $A$.      Soit $\ell$
un entier sup\'erieur ou \'egale \`a $k$.      Soit un produit droit 
$P=a_{m}a_{m-1}a_{m-2}\cdots a_{3}a_{2}a_{1}$
de longueur $m$ et de poids $n$ sup\'erieur ou \'egal \`a $2\ell$, relativement
\`a $B$.      Alors $P\in B_{(A,k)}^{\ell}$.
\end{lemm}

\begin{proof}
Le produit droit $Q=a_{m}a_{m-1}a_{m-2}\cdots a_{k+1}$
est de poids sup\'erieur ou \'egal \`a $\ell$, en effet posons $n'$ le poids
de $Q$ et $n''$ le poids de $a_{k}\cdots a_{3}a_{2}a_{1}$.      On a
$0\leq n''\leq k$ et l'\'egalit\'e $P=Qa_{k}\cdots a_{3}a_{2}a_{1}$
nous donne $n'\leq n\leq k+n'$.      Ainsi $n'\geq n-k\geq2\ell-k\geq\ell$.     
Le lemme \ref{Bn} nous dit que $Q\in B_{\ell}$.      Ainsi $P=Qa_{k}\cdots a_{3}a_{2}a_{1}\in\left(B_{\ell}\right)_{(A,k)}=\left(B^{\ell}+J(B,A,A)\right)_{(A,k)}=B^{\ell}_{(A,k)}$
 car $B$ est un id\'eal \textit{$J_{k}$-nil} dans $A$.
\end{proof} 

\begin{lemm}\label{lem:2n+1} Soit $B$ un id\'eal $J_{k}$-nil de
l'alg\`ebre de Malcev $A$.      Soit $P$ un produit de poids $t\geq4k^2-2k+1$,
relativement \`a l'id\'eal $B$.      Alors $P$ est combinaison lin\'eaire de
produits normaux $Q_j$ ($P=\ds\sum_{j \textnormal{ fini}}\mu_{j}Q_{j}$)  
tels que, pour $j$ fix\'e on a $Q_{j}$ est dans $\left(B^{k}\right)_{(A,k)}$ ou comporte au moins un facteur dans $\left(B^{k}\right)_{(A,k)}$.     
\end{lemm}

\begin{proof}  Soient $k>1$ et $t\geq4k^2-2k+1$.      D'apr\`es
le lemme \ref{lem:-prod-normo}, tout produit $P$ de poids sup\'erieur
ou \'egal \`a $t$ est combinaison lin\'eaire de produits normaux de poids
sup\'erieur ou \'egal \`a $t$.      Soit $P=\ds\sum_{j}\mu_{j}Q_{j}$ o\`u $Q_{j}$
est un produit normal de poids sup\'erieur ou \'egal \`a $t$.\par      
Pour $j$
fix\'e on a $Q_{j}=S_{j,p}S_{j,p-1}\cdots S_{j,1}$ o\`u $S_{j,i}$ est
un produit droit (avec $p$ le nombre de facteurs $S_{j,i}$ de $Q_{j}$).      
\begin{description}\label{desc}
\item[-] Si un des
produits droits $S_{j,i_{0}}$ poss\`edent un poids sup\'erieur ou \'egal
\`a $2k$, alors $S_{j,i_{0}}$ appartient \`a $\left(B^{k}\right)_{(A,k)}$
d'apr\`es le lemme \ref{lem:laqt}.     Et alors $Q_{j}$ poss\`ede un facteur
dans $\left(B^{k}\right)_{(A,k)}$. 
      
\item[-] Sinon, tous les facteurs $S_{j,i}$ poss\`edent un poids strictement
inf\'erieur \`a $2k$.      Soit $q$ le nombre de facteurs $S_{j,i}$
ayant un poids sup\'erieur ou \'egal \`a $1$.

On a alors $q\left(2k-1\right)\geq t=4k^2-2k+1$\\
Par suite  $q>2k$. Rempla\c cons chacun des produits $S_{j,i}$ par
sa valeur $s_{j,i}=S_{j,i}\in A$.  Lorsque le poids de $S_{j,i}$ est
sup\'erieur ou \'egal \`a $1$, on a $s_{j;i}\in B$.     
Ainsi $Q_{j}$ se met sous la forme $Q_{j}$ se met sous la forme d'un produit droit $s_{j,p}s_{j,p-1}\cdots s_{j,1}$
de longueur $p$ et de poids $q$ dans $B$. 
Comme $q\geq2k$, on a $Q_{j}\in\left(B^{k}\right)_{(A,k)}$
d'apr\`es le lemme \ref{lem:laqt}. 
\end{description}       
Cela ach\`eve la d\'emonstration.
\end{proof}\par 
Par un abus de language, on dira que $Q_j$ poss\`ede un facteur dans $\left(B^{k}\right)_{(A,k)}$ si (cf. les cas cit\'es dans la preuve pr\'ec\'edente) :
\begin{description}
\item[-]  un des
produits droits $S_{j,i_{0}}$  appartient \`a $\left(B^{k}\right)_{(A,k)}$,
\item[-] on a $Q_{j}\in\left(B^{k}\right)_{(A,k)}$.
\end{description}

\section{Th\'eor\`eme principal}\label{mainresult}

\begin{theo}\label{th1} Soient $K$ un corps commutatif de caract\'eristique diff\'erente
de 2, $A$ une $K$-alg\`ebre de Malcev et 
 $B$ un id\'eal $J_{k'}$-nil de Malcev $A$.      Alors les conditions
suivantes suivantes sont \'equivalentes : 
\begin{description}
\item [{(i)}] $B$ est \textit{nilpotent \`a droite} ; 
\item [{(ii)}] $B$ est \textit{nilpotente} ;
 
\item [{(iii)}] $B$ est \textit{fortement nilpotente}.      ;
\end{description}
\end{theo}
\begin{proof} En effet, pour tout entier $k>1$, les inclusions d'espaces
vectoriels\\ $B^{k}\subseteq B^{\{k\}}\subseteq B^{\left\langle k\right\rangle }$
nous montrent que $(iii)\Rightarrow(ii)\Rightarrow(i)$.     \\
 Par ailleurs, supposons qu'il existe un entier $\ell'\geq1$ tel que
$B^{\ell'}=\left\{ 0\right\} $.      Posons $k=\max\left\{ k',\ell'\right\} $,
alors pour $\ell=4k^2-2k+1$, le lemme \ref{lem:2n+1} nous dit
que tout produit $P$ de poids $\geq \ell=4k^2-2k+1$, relativement
\`a l'id\'eal $B$ est combinaison lin\'eaire de produits normaux dont chacun
comporte au moins un facteur dans $\left(B^{k}\right)_{(A,k)}\subseteq\left(B^{k'}\right)_{(A,k)}=\left\{ 0\right\} $.     
Il s'ensuit que $P=0$ et ainsi $B^{\left\langle \ell\right\rangle }=0$.     
Prouvant ainsi que $(i)\Rightarrow(iii)$.     
\end{proof}

\begin{coro}\label{equiv} Soient $K$ un corps commutatif de caract\'eristique diff\'erente
de 2, $A$ une $K$-alg\`ebre de Malcev, 
Les conditions suivantes suivantes sont \'equivalentes : 
\begin{description}
\item [{(i)}] $A$ est \textit{nilpotent \`a droite} ; 
\item [{(ii)}] $A$ est \textit{nilpotente} ;
 
\item [{(iii)}] $A$ est \textit{fortement nilpotente}.      
\end{description} 
\end{coro}
\begin{proof} Remarquons d'abord que l'id\'eal $J\left(A,A,A\right)$
de $A$ est inclus dans $A^{3}$. Supposons que la condition (i) est v\'erifi\'ee et montrons (iii).\\
Il existe donc un entier $\ell'>1$ tel que  $\{0\}=A^{\ell'}$, par suite \\
$J(A,A,A)\underset{\ell'\textnormal{ facteurs}}{\underbrace{AAA\dots A}}\subseteq
(A^3)\underset{\ell'\textnormal{ facteurs}}{\underbrace{AAA\dots A}}\subseteq A^{\ell'+3}=\{0\}$. Ainsi $A$ est $J_{\ell'}-nil$. En s'appuiant sur la preuve du th\'eor\`eme \ref{th1}, il vient que $(i)\Rightarrow (iii)$.\\
Bien entendu, $(iii)\Rightarrow(ii)\Rightarrow(i)$ du fait des inclusion de sous espaces vectoriels : $B^{k}\subseteq B^{\{k\}}\subseteq B^{\left\langle k\right\rangle}$.
\end{proof}

\end{document}